\documentclass[12pt]{amsart}
\usepackage{amsfonts}
\usepackage{amsmath}
\usepackage{amssymb}
\usepackage{graphicx}
\usepackage{enumerate}
\usepackage{multicol}
\usepackage{xcolor}
\usepackage{mathrsfs}
\usepackage[usenames,dvipsnames]{pstricks}
\usepackage{epsfig}
\usepackage[hmargin=3cm,vmargin=2.5cm]{geometry}
\usepackage{lscape}
\usepackage{epsfig}
\usepackage{verbatim}
\usepackage[utf8]{inputenc}
\usepackage{tikz}
\usepackage{tikz-cd}
\usepackage{mathdots}

\numberwithin{equation}{section}

\newcounter{dummy} \numberwithin{dummy}{section}

\newtheorem{proposition}[dummy]{Proposition}
\newtheorem{theorem}[dummy]{Theorem}
\newtheorem{corollary}[dummy]{Corollary}
\newtheorem{lemma}[dummy]{Lemma}
\theoremstyle{remark}
\newtheorem{example}[dummy]{Example}
\newtheorem{remark}[dummy]{Remark}
\newcommand{\calD}{\mathcal{D}}
\newcommand{\calL}{\mathcal{L}}
\newcommand{\calV}{\mathcal{V}}

\DeclareMathOperator{\spn}{span}
\DeclareMathOperator{\Id}{Id}
\DeclareMathOperator{\Ad}{Ad}
\DeclareMathOperator{\SO}{SO}
\DeclareMathOperator{\Ort}{O}
\DeclareMathOperator{\so}{\mathfrak{so}}

\DeclareMathOperator{\Ann}{Ann}
\DeclareMathOperator{\Hol}{Hol}
\DeclareMathOperator{\pr}{pr}
\DeclareMathOperator{\ptr}{/ \! /}
\DeclareMathOperator{\vl}{vl}
\newcommand{\bnabla}{\blacktriangledown}

\usepackage[arrow, matrix, curve, xdvi, dvips]{xy}

\title[Submersions and curves of constant geodesic curvature]{Submersions and curves of constant geodesic curvature}
\author[Godoy Molina, Grong and Markina]{Mauricio Godoy Molina, Erlend Grong and Irina Markina}
\address{Departamento de Matem\'atica y Estad\'istica, Universidad de La Frontera, Chile}
\email{mauricio.godoy@ufrontera.cl}
\address{Universit\'e Paris Sud, Laboratoire des Signaux et Syst\`emes (L2S) Sup\'elec, CNRS, Universit\'e
Paris-Saclay, 3 rue Joliot-Curie, 91192 Gif-sur-Yvette, France and University
of Bergen, Department of Mathematics, P. O. Box 7803, 5020 Bergen, Norway.}
\email{erlend.grong@gmail.com}
\address{Department of Mathematics, University of Bergen, Norway.}
\email{irina.markina@uib.no}
\thanks{The first author is partially supported by the Grants Anillo ACT 1415 PIA CONICYT and DI17-0147 from Universidad de La Frontera. The first and the last authors were partially supported by EU FP7 IRSES program STREVCOMS, grant no. PIRSES-GA-2013-612669, as well as the second and the last authors were partially supported by ISP project 239033/F20 of Norwegian Research Council. The second author was  also partially supported by FRINATEK project 249980/F20 Norwegian Research Council.}



\begin{document}

\maketitle

\begin{abstract}
Considering Riemannian submersions, we find necessary and sufficient conditions for when sub-Riemannian normal geodesics project to curves of constant first geodesic curvature or constant first and vanishing second geodesic curvatures. We describe a canonical extension of the sub-Riemannian metric and study geometric properties of the obtained Riemannian manifold. This work contains several examples illustrating the results.
\end{abstract}

\maketitle

\section{Introduction}

The study of curves in surfaces having constant geodesic curvature is an old problem in differential geometry, whose origin can be traced back to classic works by Bianchi and Darboux~\cite{B,D}. The analogous definition of geodesic curvatures for curves in a Riemannian manifold, and its relation to the generalized Frenet frame, has been known for many years, see~\cite{S} and Section~\ref{ssec:Frenet} of the present paper. The problem of determining which curves have constant geodesic curvature in the more general setting of manifolds of dimension three or higher is much more complicated, and to our knowledge there has been no comprehensive treatment of this problem. In the past few years, curves of constant geodesic curves have played an important role in interpolation in Riemannian manifolds, see e.g.~\cite{KMSB}.

In many examples in sub-Riemannian geometry, curves of constant geodesic curvature appear as images under submersions of normal sub-Riemannian geodesics. To name a few cases in which this situation occurs, the projections to the $xy$ plane of sub-Riemannian geodesics in the Heisenberg group are circles~\cite{M}, in $H$-type groups they are the circles in the horizontal layer~\cite{CChM}, and the Hopf fibration maps sub-Riemannian geodesics in the three-dimensional sphere $S^3$ to parallel circles in $S^2$, see~\cite{CMV}.

In the present paper, we give a characterization of the submersions from a sub-Riemannian manifold $M$ to a Riemannian manifold $N$ that map normal sub-Riemannian geodesics to curves with constant geodesic curvature. The main results are stated in Theorems~\ref{th:main1} and~\ref{th:main2} in Section~\ref{sec:maintheorems} in the terms defined on the underlying  Riemannian manifold $N$. To prove the result we use a special choice of the connection on the sub-Riemannian manifold $M$ and find a criterion for the main theorems in terms of this connection. In Section~\ref{ssec:spaces} we extend canonically the sub-Riemannian metric on $M$ to a Riemannian metric and study this extended Riemannian geometry. In Section~\ref{sec:example} give several examples reflecting this geometry. We are particularly focused on principal bundles and complete Riemannian submersions. Finally, we give a result for $H$-type manifolds in Section~\ref{sec:HType}.

%%%%%%%%%%%%%%%%%%%%%%%%%%%%%%%%%

\section{Submersions and sub-Riemannian geometry}\label{sec:proof}

%%%%%%%%%%%%%%%%%%%%%%%%%%%%%%%%%

In what follows, all manifolds under consideration are connected. Given a vector bundle $E\to M$ over a manifold $M$, we denote its sections by $\Gamma(E)$. All curves are defined on an interval $I =[0,\tau]$.

%%%%%%%%%%%%%%%%%%%%%%%%%%%%%%%%%

\subsection{Sub-Riemannian structures and geodesics}\label{ssec:prelim} 

%%%%%%%%%%%%%%%%%%%%%%%%%%%%%%%%%

A sub-Riemannian manifold is a pair $(M,g^*)$, where $g^*$ is a bilinear positive semidefinite tensor on the cotangent bundle $T^*M$ that degenerates along a subbundle. The latter requirement implies that the image of the map
$$
\sharp^{g}: T^*M \to T^{**}M\cong TM, \qquad\lambda\mapsto g^*(\lambda,\cdot),
$$
is a subbundle $\calD$ of $TM$. We will call $g^*$ the sub-Riemannian cometric and $\calD$ the horizontal bundle. The cometric induces a positive definite tensor $g$ on the subbundle ${\mathcal D}$ by the relation
\begin{equation}\label{eq:metric}
g(\sharp^{g}\lambda_1,\sharp^{g}\lambda_2):=g^*(\lambda_1,\lambda_2), \quad\lambda_1,\lambda_2\in T_x^*M,\quad x\in M.
\end{equation}
Conversely, given a pair $({\mathcal D},g)$, where $\mathcal D$ is a subbundle of the tangent bundle $TM$ and $g$ is a metric on ${\mathcal D}$, the relation~\eqref{eq:metric} defines a cometric $g^*$ degenerating along the subbundle ${\rm Ann}({\mathcal D})$ of covectors vanishing on~${\mathcal D}$. Hence a sub-Riemannian manifold can be equivalently defined as a triple $(M,{\mathcal D},g)$ consisting of a subbundle ${\mathcal D}$ of $TM$ endowed with a metric $g$ defined on ${\mathcal D}$, see~\cite{M}. In what follows, we write $g(v,w) = \langle v,w \rangle_g$ and $|v|_g = \langle v,v \rangle_g^{1/2}$ for any $v,w\in \calD_x$. We use similar notation related to $g^*$, though we note that for a fixed $x \in M$, $| \, \cdot \, |_{g^*}$ is only a semi-norm on $T_x^*M$. The upper-case letters $V,W$ are used to denote sections of vector bundles and lower-case letters $v,w$ denote vectors. 

We assume that the subbundle $\calD$ is bracket generating, i.e. the sections of $\calD$ and their iterated brackets span $TM$ at each point $x\in M$. In this case the Carnot-Carath\'eodory distance function on $M$ is defined by taking the infimum over the length of curves tangent to $\calD$ and connecting given points. A curve is called a sub-Riemannian geodesic if any sufficiently short segment of the curve realizes this infimum. The sub-Riemannian geodesics are divided into two groups: normal and abnormal. An abnormal sub-Riemannian geodesic is the shortest curve in the class of abnormal curves whose properties are related only to the nature of the subbundle $\mathcal D$, see for instance~\cite{Monti, M}. They will not be considered in the present work. The normal sub-Riemannian geodesics are the projections on $M$ of solutions to the Hamiltonian system generated by the sub-Riemannian Hamiltonian function $H(\lambda)=\frac{1}{2}|\lambda|_{g^*}^2$. They can be described by the following statement.

\begin{lemma}\cite[Proposition 2.1]{GG} \label{lemma:GG}
Let $\nabla$ be any affine connection on $M$ such that $\nabla g^*=0$. A curve $\gamma\colon I\to M$ is a normal sub-Riemannian geodesic if and only if there exists a covector field $\lambda(t)$ along $\gamma(t)$ which satisfies $\sharp^g\lambda(t)=\dot\gamma(t)$ and
\begin{equation}\label{eq:srgeodeqT}
\nabla_{\dot\gamma}\lambda(t)= -\lambda(t)T\big(\dot\gamma(t), \, \cdot \, \big),
\end{equation}
where $T(V,W)=\nabla_VW-\nabla_WV-[V,W]$ is the torsion of the connection $\nabla$.
\end{lemma}
A connection satisfying $\nabla g^* = 0$ will be called \emph{compatible} with the sub-Riemannian structure. We call the covector field $\lambda(t)$ along the geodesic described above \emph{a sub-Riemannian extremal}. The curve $\lambda(t) = e^{t\vec{H}}(\lambda(0))$ is an integral curve of the sub-Riemannian Hamiltonian vector field $\vec{H}$ associated to the sub-Riemannian Hamiltonian function $H$. When proving statements about all normal geodesics, it is usually sufficient to consider the case $t = 0$ from the property $\lambda(s+ t) = e^{t\vec{H}}(\lambda(s))$. This fact will be used in the proof of the main results.

%%%%%%%%%%%%%%%%%%%%%%%%%%%%%%%%%%

\subsection{Submersions, Ehresmann connections and curvature}\label{sec:Submersion}

%%%%%%%%%%%%%%%%%%%%%%%%%%%%%%%%%%

Let $\pi: M \to N$ be a surjective submersion of manifolds, meaning that $d\pi: TM \to TN$ is a surjective map as well. If we let $m$ and $n$ denote the dimension of $M$ and $N$, then the subbundle $\calV = \ker d\pi \subseteq TM$ is called \emph{the vertical bundle} of the submersion $\pi$ and it has rank $m-n$. Since each value $y \in N$ is a regular value of $\pi$, the fiber $M_y = \pi^{-1}(y)$ is an embedded submanifold of~$M$, which is obviously tangent to $\calV|{M_y}$.

\emph{An Ehresmann connection $\calD$ on $\pi$} is a choice of a subbundle of $TM$ satisfying $TM = \calD \oplus \calV$. Since $d\pi$ restricted to $\calD_x$ gives a bijection to $T_{\pi(x)}N$ for every $x \in N$, we have its inverse map $h_x: T_{\pi(x)}N \to \mathcal{D}_x$, which is called the \emph{horizontal lift}. The map $h_x$ allows us to define horizontal lifts $hX$ of vector fields $X$ on $N$ by $hX(x) = h_x X(\pi(x))$. If we fix a point $x_ 0$ in the fiber $M_{y_0}$, then for any absolutely continuous curve $\eta\colon I \to N$ with $\eta(0)=y_0$ there is a unique curve $\gamma(t)$ almost everywhere tangent to $\mathcal{D}$ and satisfying $\pi(\gamma(t)) = \eta(t)$, $\gamma(0) = x_0$. The curve $\gamma$ is given by the solution to the differential equation
$$\dot \gamma(t) = h_{\gamma(t)} \dot \eta(t), \qquad \gamma(0) = x_0.$$

Let $\pr_\calD$ and $\pr_\calV$ be the projections to $\calD$ and $\calV$, respectively, associated to the decomposition $TM = \calD \oplus \calV$. \emph{The curvature $R^\calD$} of $\calD$ is a vector-valued two-form given by
$$R^{\calD}(V,W) = \pr_\calV [\pr_\calD V, \pr_\calD W], \qquad V,W \in \Gamma(TM).$$
Notice that $R^\calD = 0$ if and only if $\calD$ is an integrable subbundle. For vectors $v,w \in T_y N$, define $R^\calD(v,w)$ as the vector field on $M_y$ by $x \mapsto R^\calD(h_x v, h_x w)$.

We will need the following notion developed in~\cite{Gro16}. Let $\Ann(\calD)$ denote the subbundle of $T^*M$ of covectors vanishing on $\calD$ and let $\pi_{T^*M}\colon T^*M\to M$ be the canonical projection. Let $\Pi$ be the projection 
$$\Pi\colon \Ann(\calD) \to M \to N, \quad\text{that is}\quad \Pi=\pi_{T^*M}\circ\pi.$$
Even though $\Pi$ is not a vector bundle, we can define an analogue of parallel transport of elements along a curve $\eta\colon I \to N$. Define
$$\bnabla_X \beta = \pr_\calV^* \calL_{hX} \beta, \qquad X \in \Gamma(TN),\ \  \beta \in \Gamma(\Ann(\calD)).$$
Here $\pr_\calV^*$ is the pullback of the vertical projection, i.e., $(\pr_\calV^* \alpha)(v) = \alpha(\pr_\calV v)$.
Note that the equalities $\bnabla_{fX} \beta = f \bnabla_X \beta$ and $\bnabla_X \tilde f \beta =  (hX\tilde f) \beta + \tilde f \bnabla_X \beta$ are valid for any $f \in C^\infty(N)$ and $\tilde f \in C^\infty(M)$. Furthermore, if $\gamma\colon I \to M$ is the horizontal lift of a curve $\eta$ and $\beta(t)$ is a section of $\Ann(\calD)$ along $\gamma(t)$, then $\bnabla_{\dot \eta} \beta$ is well defined. Observe that if $\beta_0 \in \Ann(\calD)_{y_0} = \Pi^{-1}(y_0)$, then there is a unique $\beta(t)$ solving the equation
$$\bnabla_{\dot \eta} \beta(t) = 0, \qquad \beta(0) = \beta_0.$$

%%%%%%%%%%%%%%%%%%%%%%%%%%%%%%%%

\subsection{Frenet frame and geodesic curvatures}\label{ssec:Frenet}

%%%%%%%%%%%%%%%%%%%%%%%%%%%%%%%%

We recall the notion of geodesic curvatures for a curve $\eta$ in an $n$-dimensional Riemannian manifold $N$. For more details, see \cite[Appendix~B]{S}. Let $g_N(\cdot\,,\cdot)=\langle\cdot\,,\cdot\rangle_{g_N}$ be a Riemannian metric and $\nabla^N$ the Levi-Civita connection on~$N$. We assume that the curve $\eta\colon I\to N$ is parameterized by arc length and define a unit vector field by $e_1(t) = \dot \eta(t)$. We call it the first Frenet vector field. Next, we define the first geodesic curvature by
 \begin{equation}\label{eq:kappa1CL}
\kappa_1(t)=|\nabla^N_{\dot\eta}\dot\eta(t)|_{g_N}.
\end{equation}
The curve $\eta(t)$ is a geodesic if and only if $\kappa_1(t)$ vanishes identically. If $\kappa_1(t) \neq 0$, we define the second Frenet vector field $e_2$ as the unit vector field determined by relation
\begin{equation} \label{eq:Kappa2}
\nabla_{\dot \eta}^N \dot \eta(t) = \nabla_{\dot \eta}^N e_1(t) = \kappa_2(t) e_2(t).\end{equation}
Since $\langle e_1(t),\nabla^N_{\dot\eta} e_1(t)\rangle_{g_N}=0$, we have that $e_2$ is orthogonal to $e_1$. Continuing inductively, assuming that $e_1(t), \dots, e_j(t)$ and $\kappa_1(t), \dots, \kappa_{j-1}(t)$ have been defined, we can define $j$-th geodesic curvature by
$$\kappa_j(t) = |\nabla^N_{\dot \eta} e_j(t) + \kappa_{j+1}(t) e_{j+1}|_{g_N},$$
and if $\kappa_j(t)$ is the non-vanishing, then we define $(j+1)$-st Frenet vector field by relation
\begin{equation} \label{eq:KappaJ}
\nabla_{\dot \eta} e_j(t) = - \kappa_{j-1}(t) e_{j-1}(t) + \kappa_j(t) e_{j+1}(t).\end{equation}
We call $e_1(t), \dots, e_n(t)$ the Frenet frame along $\eta$, provided it exists.
If all geodesic curvatures $\kappa_1(t)$, $\dots$, $\kappa_{n-1}(t)$ are well-defined, they uniquely determine the curve $\eta$ up to initial conditions $\eta(0)$ and $\dot \eta(0)$. The same is true if $\kappa_1(t), \dots, \kappa_{j-1}(t)$ is well defined and $\kappa_j$ vanish identically. In this case, we define all higher geodesic curvatures to be vanishing.

We will focus on curves with constant first geodesic curvatures or with constant first geodesic curvature and vanishing second geodesic curvatures. When $N$ is the Euclidean space, the latter curves are circle arcs.

%%%%%%%%%%%%%%%%%%%%%%%%%%%%%%%%%%

\subsection{Projections of geodesics}\label{sec:maintheorems}

%%%%%%%%%%%%%%%%%%%%%%%%%%%%%%%%%%

Let $\pi\colon M \to N$ be a submersion into a Riemannian manifold $(N, g_N)$. Consider a sub-Riemannian manifold $(M, \calD, g)$, where the sub-Riemannian metric $g$ is the pullback of the metric $g_N$ to $\calD$. We use the parallel transport defined by $\bnabla$ to state a criterion for a curve in $M$ to be a normal sub-Riemannian geodesic in terms of its projection under $\pi$. For any element $\beta \in \Ann(\calD)$ with $\Pi(\beta) = y$, we define a linear map $J_{\beta}\colon T_y N \to T_y N$ by
$$\langle J_{\beta} v, w \rangle_{g_N} = \beta R^{\calD}(v,w).$$  
Note that $J_{\beta}^* = - J_{\beta}$ by the skew symmetry of the curvature $R^{\calD}$.
\begin{lemma}\cite[Corollary~2.3]{Gro16}\label{lem:Gro16}
A curve $\gamma: I \to M$ is a sub-Riemannian normal geodesic if and only if it is the horizontal lift of a curve $\eta: I \to N$ satisfying
\begin{equation}\label{eq:nablaeta}
\nabla_{\dot \eta}^N \dot\eta = J_{\beta(t)} \dot \eta, \qquad \bnabla_{\dot \eta} \beta(t) = 0,
\end{equation}
for some section $\beta(t)$ of $\Ann(\calD)$ along $\eta(t)$.
\end{lemma} 

We now state the main results of the present paper. We ask the reader to pay special attention to the fact that the below statements both involve sub-Riemannian geodesics in $M$ and Riemannian geodesics in $N$. 
\begin{theorem} \label{th:main1}
The following statements are equivalent.
\begin{enumerate}[\rm (I')]
\item \label{item:Iprime} The projection of any normal sub-Riemannian geodesic $\gamma$ to $N$ is a curve with constant first geodesic curvature.
\item \label{item:IIprime} For any Riemannian geodesic $\eta$ in $N$ and any covector field $\beta\in\Gamma\big(\Ann(\calD)\big)$ along $\eta$ satisfying $\bnabla_{\dot \eta} \beta = 0$, the vector field $J_{\beta(t)} \dot \eta$ has constant length.
\end{enumerate}
\end{theorem}

\begin{theorem} \label{th:main2}
The following statements are equivalent.
\begin{enumerate}[\rm (I)]
\item \label{item:I} The projection of any normal sub-Riemannian geodesic $\gamma$ to $N$ is a curve with constant first geodesic curvature and vanishing second geodesic curvature.
\item \label{item:II}For any Riemannian geodesic $\eta$ in $N$ and any covector field $\beta\in\Gamma\big(\Ann(\calD)\big)$ along $\eta$ satisfying $\bnabla_{\dot \eta} \beta = 0$, the vector field $J_{\beta(t)} \dot \eta$ is parallel. Furthermore, for any $\alpha \in \Ann(\calD)_x$ and $v \in T_y N$ satisfying $|v|_{g_N} =1$, we have
\begin{equation} \label{eq:J2} J_{\alpha}^2 v = - |J_{\alpha} v|^2 v.
\end{equation}
\end{enumerate}
\end{theorem}

%%%%%%%%%%%%%%%%%%%%%%%%%%%%%%%%%%

\subsection{Description of geodesic curvatures}

%%%%%%%%%%%%%%%%%%%%%%%%%%%%%%%%%%

Before we proceed to the proof of the main result we give some additional information about the tools used in the proof. We make a special choice of connection $\nabla$ on the manifold $M$, that will help us to perform the technical calculations. From now on, we simplify notation and write $R = R^\calD$.
Let $\nabla^N$ denote the Levi-Civita connection of $g_N$. We identify the two vector bundles ${\mathcal D}$ and $\pi^*TN$ over $M$ via the horizontal lift, described in Section~\ref{sec:Submersion}. Choose an affine connection $\nabla$ on $TM$ satisfying the following three conditions:
\begin{enumerate}[\rm (i)]
\item $\nabla_VW=(\pi^*\nabla^N)_VW$ for any $V\in\Gamma(TM)$ and $W\in\Gamma(\mathcal D)$, 
\item $\nabla g^*=0$,
\item $\nabla_VW={\rm pr}_{\mathcal V}[V,W]$ for $V\in\Gamma({\mathcal D})$ and $W\in\Gamma({\mathcal V})$.
\end{enumerate}
Here, we have identified $\pi^* TN$ with $\calD$ using the horizontal lift. With this identification, if $\nabla^N$ is an affine connection on $TN$, the pullback connection $\pi^*\nabla^N$ on $\calD$ is the unique connection satisfying
$$(\pi^*\nabla^N)_{hX} hY = h \nabla_X^N Y, \qquad (\pi^* \nabla^N)_V hX = 0,$$
for any $V \in \Gamma(\calV)$ and $X,Y \in \Gamma(TN)$.

The class of connections satisfying conditions {\rm (i)-(iii)} is not empty. A well-known connection that fulfills all the above requirements is the Bott connection associated to the subbundle ${\mathcal D}$ with any choice of an extension of the sub-Riemannian metric $g$ to a Riemannian metric $g_M$ that makes ${\mathcal D}$ and ${\mathcal V}$ orthogonal, see~\cite{Bott,GG} and Section~\ref{ssec:spaces}.

In the following proposition we collect the properties and useful consequences of the definition of the connection $\nabla$, which we leave to the reader to verify. 

\begin{proposition}\label{prop:nabla}
Let $\nabla$ be a connection on $M$ defined by (i)-(iii) and let the curve $\gamma\colon I\to M$ be the horizontal lift of $\eta\colon I\to N$. Then the following results hold:
\begin{enumerate}[\rm (a)]
\item \label{item:a}{$\nabla_VW={\rm pr}_{\mathcal D}[V,W]$ for any $V\in\Gamma(\mathcal V)$ and $W\in\Gamma(\mathcal D)$;}
\item \label{item:b}{The connection $\nabla$ and the operator $\sharp^g$ commute;}
\item \label{item:c}{ $\bnabla_{\dot \eta} \beta =0$ if and only if $\nabla_{\dot \gamma} \beta = 0$ for any $\beta \in \Gamma(\Ann(\calD))$;}
\item \label{item:d}{The torsion $T$ of $\nabla$ satisfies: $T(V,W)=0$ whenever $V\in\Gamma({\mathcal V})$ and $W\in\Gamma({\mathcal D})$;}
\item \label{item:e}{The torsion $T$ of $\nabla$ satisfies: $T(hX,hY)=-R(hX,hY)$ for all $X,Y\in\Gamma(TN)$;}
\item \label{item:f}{If $\gamma$ is a normal sub-Riemannian geodesic and $\sharp^g\lambda=\dot\gamma$, then one form $\nabla_{\dot\gamma(t)}\lambda(t)$, $t\in I$, vanishes on vertical vectors for all $t\in I$;}
\item \label{item:g}{The geodesic equation~\eqref{eq:srgeodeqT} can be written as $\nabla_{\dot\gamma}\lambda(t)=\lambda(t)R\big(\dot\gamma(t),\cdot\big).$}
\end{enumerate}
\end{proposition}

The fact that the curve $\eta$ is the projection of a normal sub-Riemannian geodesic $\gamma$ under the Riemannian submersion $\pi\colon M\to N$, allows us to express the first and the second geodesic curvatures of $\eta$ in terms of the chosen connection $\nabla$ on $M$.   

\begin{lemma} \label{lemma:Computation}
Let $\gamma(t)$ be a normal sub-Riemannian geodesic, parametrized by arc length, corresponding to the extremal $\lambda(t)$. Let $\eta(t)$ be its projection to $N$, and let $\kappa_1(t)$ and $\kappa_2(t)$ be the first and the second geodesic curvatures of $\eta(t)$. Then
$$\kappa_1(t) = | \lambda(t) R(\dot \gamma, \, \cdot \, ) |_{g^*} ,$$
and at any point where $\kappa_1(t) \neq 0$,
$$\kappa_2(t) = \frac{1}{\kappa_1(t)}\Big| \lambda(t) (\nabla_{\dot \gamma} R)(\dot \gamma(t) , \, \cdot \, ) + \lambda(t) R(\nabla_{\dot \gamma} \dot \gamma(t) , \, \cdot \,) - \frac{\dot \kappa_1(t)}{\kappa_1(t)}  \lambda(t) R(\dot \gamma(t) , \, \cdot \,) + \kappa_1(t)^2 \lambda(t)\Big|_{g^*}.$$
\end{lemma}
\begin{proof}
Define $e_1(t) = \dot \eta(t)$ and observe that $|\dot\gamma(t)|_{g}=|\dot \eta(t)|_{g_N}=1$. Then
\begin{eqnarray*}
d\pi \nabla_{\dot \gamma} \dot \gamma(t) 
& = & \nabla_{\dot \eta}^N e_1(t) = \kappa_1(t) e_2(t), 
\\
d\pi \nabla_{\dot \gamma} \nabla_{\dot \gamma} \dot \gamma(t)
& = & \nabla_{\dot \eta}^N \kappa_1(t) e_2(t) = \dot \kappa_1(t) e_2(t) + \kappa_1(t) \kappa_2(t) e_3(t) - \kappa_1(t)^2 e_1(t),
\end{eqnarray*}
by the definition of the pullback connection and equations~\eqref{eq:Kappa2} and \eqref{eq:KappaJ}.
Using the geodesic equation in the form \eqref{item:g} from Proposition~\ref{prop:nabla}, we have
$$\nabla_{\dot \gamma} \dot \gamma(t) = \sharp^g \nabla_{\dot \gamma} \lambda(t) = \sharp^g \lambda(t) R(\dot \gamma(t), \, \cdot \,).$$
and
\begin{equation}\label{eq:ggg}
\nabla_{\dot \gamma} \nabla_{\dot \gamma} \dot \gamma(t) =  \sharp^g \lambda(t) (\nabla_{\dot \gamma} R)(\dot \gamma(t) , \, \cdot \, ) + \sharp^g \lambda(t) R(\nabla_{\dot \gamma} \dot \gamma(t) , \, \cdot \,).
\end{equation}
In the latter equation, we have used that $\sharp^g (\nabla_{\dot \gamma} \lambda)(t) R(\dot \gamma(t) , \, \cdot \, ) = 0$ since $\nabla_{\dot \gamma} \lambda(t)$ vanishes along $\calV$ by~property \eqref{item:d} from Proposition~\ref{prop:nabla}.

It follows from the definitions of geodesic curvature, pullback connection, cometric, and the properties \eqref{item:c} and \eqref{item:g} that
$$
\kappa_1(t) = \left| \nabla^N_{\dot\eta}\dot\eta\right|_{g_N}
=\left|\nabla_{\dot\gamma}\dot\gamma\right|_{g}
=\left|\sharp^{g}\nabla_{\dot\gamma}\dot\gamma\right|_{g^*}
=\left|\nabla_{\dot\gamma}\lambda\right|_{g^*}
=\left| \lambda(t) R( \dot \gamma(t), \,\cdot \, )\right|_{g^*}
$$
At any point where $\kappa_1(t) \neq 0$, we can define $e_2$ by
\begin{equation}\label{eq:e2}
h_{\gamma(t)} e_2(t) = \frac{1}{\kappa_1(t)} \sharp^g\lambda(t) R( \dot \gamma(t), \, \cdot \, ).
\end{equation}
 
Finally, we have
\begin{equation} \label{eq:Kappa1Kappa2}\kappa_1(t) \kappa_2(t) h_{\gamma(t)} e_3(t) =  \nabla_{\dot \gamma} \nabla_{\dot \gamma} \dot \gamma(t) - \dot \kappa_1(t) h_{\gamma(t)} e_2(t) + \kappa_1(t)^2h_{\gamma(t)} e_1.\end{equation}
The result follows from~\eqref{eq:ggg} and~\eqref{eq:e2}.
\end{proof}

The following lemma gives the necessary and sufficient criterion for statement (I') of Theorem~\ref{th:main1} in terms of the connection $\nabla$.

\begin{lemma} \label{lemma:Geodesic} 
The projection of any normal sub-Riemannain geodesic has constant first geodesic curvature $\kappa_1$ if and only if for any $(\alpha, v) \in \Ann(\calD_x) \oplus \calD_x$, $x\in M$, we have
\begin{equation} \label{dotKappaGeneral} 
\langle \alpha R(v, \cdot) , \alpha(\nabla_v R)(v, \, \cdot \,)\rangle_{g^*} = 0.
\end{equation}
\end{lemma}
\begin{proof} Lemma~\ref{lemma:Computation} shows that the geodesic curvatures are completely defined by the behavior of the extremal $\lambda$ and, therefore, it is enough to give the conditions only for $t=0$. We choose arbitrary $(\alpha, v) \in \Ann(\calD_x) \oplus \calD_x$, $x\in M$, $|v|_g=1$, and find the unique normal extremal $\lambda\colon I \to T^*M$  such that 
\begin{equation} \label{LambdaDet}
\lambda(0) \in T_x^*M,\qquad \mathrm{pr}_{\calV}^* \lambda(0) = \alpha, \qquad \sharp^g \lambda(0) = v. \end{equation}
Let $\gamma(t)$ be the projection of $\lambda(t)$ to $M$ and let $\eta(t)$ be the projection of $\gamma(t)$ to $N$. We want to find conditions for the first geodesic curvature $\kappa_1$ to be constant for all choices of $\alpha$ and $v$. It is clearly sufficient to look at $\dot \kappa_1(0)$.

Using Lemma~\ref{lemma:Computation}, we have that
\begin{eqnarray} \label{dotKappa}  
\kappa(t) \dot \kappa_1(t) 
& = &  \langle \lambda(t) R(\dot \gamma(t), \, \cdot \,) , \lambda(t) (\nabla_{\dot \gamma}R)(\dot \gamma(t), \, \cdot \, ) + \lambda(t) R(\nabla_{\dot \gamma} \dot \gamma(t) , \, \cdot \,) \rangle_{g^*} \\ \nonumber
&= &\langle \lambda(t) R(\dot \gamma(t), \, \cdot \,) , \lambda(t) (\nabla_{\dot \gamma}R)(\dot \gamma(t), \, \cdot \, ) \rangle_{g^*}.
\end{eqnarray}
Here, we have used that
\begin{align}\label{eq:rgrg} 
& \langle \lambda(t) R(\dot \gamma(t), \, \cdot \,) ,  \lambda(t) R(\nabla_{\dot \gamma} \dot \gamma(t) , \, \cdot \,) \rangle_{g^*}  \nonumber
\\ 
& = \lambda(t) R\Big(\nabla_{\dot \gamma} \dot \gamma(t) , \sharp^g \lambda(t) R(\dot \gamma(t), \, \cdot \,) \Big) = \lambda(t) R(\nabla_{\dot \gamma} \dot \gamma(t), \nabla_{\dot \gamma} \dot \gamma(t))  = 0. 
\end{align}
Evaluating~\eqref{dotKappa} at $t=0$ we finish the proof, because $\gamma$ is a normal sub-Riemannian geodesic and $\sharp^g\lambda(t)=\dot\gamma(t)$.
 \end{proof}

Analogously we formulate the criterion for statement (I) in Theorem~\ref{th:main2} by using the connection $\nabla$.

\begin{lemma}\label{lemma:Geodesic_b} The projection of any normal sub-Riemannain geodesic has constant first geodesic curvature $\kappa_1$ and vanishing second geodesic curvature $\kappa_2=0$ if and only if for any $(\alpha, v) \in \Ann(\calD_x) \oplus \calD_x$, $x \in M$, we have
\begin{eqnarray}\label{eq:vRv}
\alpha (\nabla_{v} R)(v , w) & = & 0, \qquad \text{for any\ \  $w \in T_xM$,}\\
\label{eq:RvRw}
\langle  \alpha R(v, \, \cdot \, ) ,\alpha R(w, \, \cdot \, ) \rangle_{g^*} & = &0, \qquad w \in \left(\spn \left\{ v, \sharp^g \alpha R(v, \, \cdot \,) \right\} \right)^\perp.
\end{eqnarray} 
\end{lemma}
\begin{proof}
If we assume that the first geodesic curvature of any projection of a normal sub-Riemannian geodesic is constant, then we obtain that $\alpha R(v, w)= \langle \sharp^g \alpha R(v, \, \cdot \, ), w \rangle_{g} = 0$ for any $w$ parallel to $\sharp^g \alpha R(v, \, \cdot \,)$ by Lemma~\ref{lemma:Geodesic}. Hence, we only need to prove that \eqref{eq:vRv} and \eqref{eq:RvRw} hold for $w \in \left(\spn \left\{ v, \sharp^g \alpha R(v, \, \cdot \,) \right\} \right)^\perp$.

Choose an arbitrary $(\alpha, v) \in \Ann(\calD)_x \oplus \calD_x$, $x \in M$, with $|v|_g =1$, and let $\lambda(t)$ and $\gamma(t)$ be the extremal and the normal sub-Riemannian geodesic, respectively, as determined by~\eqref{LambdaDet}. Let $\kappa_1$ and $\kappa_2$ be the first and second geodesic curvatures of the projection $\eta$ of $\gamma$. We only need to consider the case $t=0$, taking an arbitrary $w\in T_xM$ orthogonal to $v$ and $\sharp^g \alpha R(v, \, \cdot \,)$. We choose an orthonormal basis $v_1, \dots, v_n$ of $\calD_x$. Note the identities, 
\begin{eqnarray*}
\nabla_{\dot \gamma} \nabla_{\dot \gamma} \dot \gamma(0) & = & \sharp^g \alpha (\nabla_{v} R)(v , \, \cdot \, ) + \sum_{i=1}^n \alpha R(v, v_i) \sharp^g \alpha R(v_i , \, \cdot \,) ,\\
\dot \kappa_1(0) h_{x} e_2 & = & \frac{\dot \kappa_1(0)}{\kappa_1(0)} \sharp^g \alpha R(v, \, \cdot \,), \\
\kappa_1(0)^2 h_{x} e_1 & = & \kappa_1(0)^2 v,
\end{eqnarray*}
obtained by evaluating~\eqref{eq:ggg} and~\eqref{eq:e2} at $t=0$. By using~\eqref{eq:Kappa1Kappa2}, we have that $\kappa_2(0) = 0$ if and only if
\begin{equation} \label{Comp1} \sharp^g \alpha (\nabla_{v} R)(v , \, \cdot \, ) + \sum_{i=1}^n \alpha R(v, v_i) \sharp^g \alpha R(v_i , \, \cdot \,) - \frac{\dot \kappa_1(0)}{\kappa_1(0)} \sharp^g \alpha R(v, \, \cdot \,) + \kappa_1(0)^2 v =0. \end{equation}
Next, observe that both $\kappa_1(0)$ and $\dot \kappa_1(0)$ will have the same value if we change $\alpha$ to $-\alpha$. Evaluating \eqref{Comp1} at $(\alpha, v)$ and $(-\alpha, v)$, and taking sums and differences, we get
$$-\sharp^g \alpha (\nabla_{v} R)(v , \, \cdot \, ) + \frac{\dot \kappa_1(0)}{\kappa_1(0)} \sharp^g \alpha R(v, \, \cdot \,)  =0,$$
$$ \sum_{i=1}^n \alpha R(v, v_i) \sharp^g \alpha R(v_i , \, \cdot \,) + \kappa_1(0)^2 v =0.$$
Evaluating at some $w\in \left(\spn \left\{ v, \sharp^g \alpha R(v, \, \cdot \,) \right\} \right)^\perp$, we obtain the result.
\end{proof}

%%%%%%%%%%%%%%%%%%%%%%%%%%%%%%%%%%

\subsection{Proof of Theorem~\ref{th:main1}}

%%%%%%%%%%%%%%%%%%%%%%%%%%%%%%%%%%

By Lemma~\ref{lemma:Geodesic}, Item~(I') is equivalent to \eqref{dotKappaGeneral}, so we only need to show that the latter is equivalent to Item~(II'). Since we are considering only geodesics and $\bnabla$-parallel transports, is sufficient to consider property (II') only at $t=0$.

For an arbitrary element $\alpha \in \Ann(\calD)_x$ and $v \in T_y N$, $\pi(x) = y \in N$, let $\eta$ be the geodesic with initial condition $\eta(0) = y$ and $\dot \eta(0) = v$. Let $\gamma$ be the horizotal lift of $\eta$ to $x$ and define $\beta(t)$ along $\gamma$ by
$$
\bnabla_{\dot\eta}\beta=0,\quad\beta(0)=\alpha.
$$
The condition $\nabla^N_{\dot\eta(t)}{\dot\eta}(t)=0$ implies $\nabla_{\dot\gamma(t)}\dot\gamma(t)=0$ by the definition of the pullback connection. 
The definition of the operator $J_{\beta}$ implies that $\sharp^{g}\big(\beta R(v,\cdot)\big)=hJ_{\beta}v(\cdot)$ and since $\dot\eta$, $\dot\gamma$, and $\dot\beta$ are parallel vector fields , we obtain 
\begin{eqnarray*}
\frac{1}{2} \partial_t |J_{\beta(t)} \dot \eta(t) |^2 
&= &\langle J_{\beta(t)} \dot \eta(t), \nabla_{\dot \eta}^N J_{\beta(t)} \dot \eta(t)\rangle_{g_N} \nonumber
\\ 
&=&
\langle hJ_{\beta(t)} \dot \eta(t), h\nabla_{\dot \eta}^N J_{\beta(t)} \dot \eta(t)\rangle_{g} \nonumber
\\ 
&= &\langle \beta(t) R(\dot \gamma(t), \cdot), \nabla_{\dot \gamma}\beta(t) R(\dot \gamma(t), \cdot) \rangle_{g^*} 
\\ 
&=&
\langle \beta(t) R(\dot \gamma(t), \cdot), \beta(t) (\nabla_{\dot \gamma(t)} R)(\dot \gamma(t), \cdot) \rangle_{g^*}\nonumber
\end{eqnarray*}
and
\begin{eqnarray} \label{eq:DerJ} 
\frac{1}{2} \partial_t |J_{\beta(t)} \dot \eta(t) |^2 |_{t=0} & = &
\langle \alpha R(h_xv, \cdot), \alpha(\nabla_{h_xv} R)(h_xv, \cdot) \rangle_{g^*}.
\end{eqnarray}
The equation \eqref{eq:DerJ} shows the equivalence between (II') and \eqref{dotKappaGeneral} and the proof is completed.\qed

%%%%%%%%%%%%%%%%%%%%%%%%%%%%%%%%%%

\subsection{Proof of Theorem~\ref{th:main2}}

%%%%%%%%%%%%%%%%%%%%%%%%%%%%%%%%%%
By Lemma~\ref{lemma:Geodesic_b}, we only need to show the equivalence of both \eqref{eq:vRv} and \eqref{eq:RvRw} holding and \eqref{item:II}. 
Let $\eta$ be a geodesic on $N$ and $\beta$ chosen such that $\bnabla_{\dot \eta} \beta = 0$. Let $\gamma$ be a horizontal lift of $\eta$. These curves are uniquely defined by the initial data
$$
y\in N, \ x\in M_y,\ (\alpha, h_x v)\in\Ann(\calD_x) \oplus \calD_x,\quad \eta(0)=y, \ \beta(0)=\alpha,\ \gamma(0)=x.
$$ 
For any vector $w \in T_yN$ we obtain 
\begin{eqnarray*}
\langle w,\nabla_{\dot \eta}^N J_{\beta(t)} \dot \eta(t) |_{t=0}\rangle_{g_N} 
&=&\langle h_xw, \sharp^g\beta(t) (\nabla_{\dot \gamma} R)(\dot \gamma(t), \, \cdot \,) |_{t=0} \rangle_{g}
\\
&=&
\langle h_xw, \sharp^g \alpha (\nabla_{h_x v} R)(h_x v, \, \cdot \,) \rangle_{g}=0,
\end{eqnarray*}
Hence, the property~\eqref{eq:vRv}  is equivalent to the statement that all the vector fields $\nabla_{\dot \eta}^N J_{\beta(t)} \dot \eta(t)$ are parallel. 

Finally, we let $w_1, \dots, w_k$ be an orthonormal basis of the complement of $\spn \{ v, J_\alpha v\}$, where $v$ is of unit length. Here $k=n-1$ if $J_\alpha v = 0$ and otherwise, $k=n-2$. Observe that $\langle J_\alpha^2 v, J_\alpha v \rangle_{g_N} = 0$ from the skew-symmetry of $J_\alpha$.
Hence
$$J_\alpha^2 v= \langle J_\alpha^2 v, v \rangle_{g_N} v+ \sum_{i=1}^k \langle J_\alpha^2 v, w_i  \rangle_{g_N} w_i = - | J_\alpha v|^2_{g_N} v -\sum_{i=1}^k \langle J_\alpha v, J_\alpha w_i  \rangle_{g_N} w_i.$$
Furthermore,
$$\langle J_\alpha v, J_\alpha w_i \rangle_{g_N} = \langle \alpha R(h_x v, \, \cdot \,) , \alpha R(h_x w_i , \, \cdot \, ) \rangle_{g^*},$$
and since $(\spn \{ h_x v, \sharp^g  \alpha R(h_x v, \, \cdot \, )\})^\perp = \spn\{ h_x w_1, \dots, h_x w_k\} \oplus \calV_x$, the equivalence of \eqref{eq:RvRw} and \eqref{eq:J2} follows. This result completes the proof.\qed

%% Section 2.8%%%%%%%%%%%%%%%%%%%%%%

\subsection{Projected geodesics and parallel curvature}

%%%%%%%%%%%%%%%%%%%%%%

We remark the following consequence relating results on the projections of normal geodesics and curvature.
\begin{corollary} \label{cor:Parallel}
Assume that the projection of any normal sub-Riemannian geodesic in $(M, \calD, g)$ is a curve with constant first geodesic curvature and vanishing second geodesic curvature. Then
\begin{equation} \label{eq:vR} \nabla_v R = 0 \qquad \text{for any } v \in \calD.\end{equation}
\end{corollary}

\begin{proof}
We will show that the statement $(\nabla_vR)(v,\, \cdot\, )=0$ for any $v\in \mathcal{D}$ implies that $\nabla_vR=0$ for any $v\in \mathcal{D}$. The result then follows from Lemma~\ref{lemma:Geodesic_b}.

Let $R^\nabla$ denote the curvature tensor defined by the connection $\nabla$: $R^{\nabla}(X,Y)=\nabla_X\nabla_Y-\nabla_Y\nabla_X-\nabla_{[X,Y]}$. Then the torsion $T$ of the connection $\nabla$ and $R^{\nabla}$ are related by Bianchi's first identity 
\begin{equation}\label{1}
\mathfrak S R^{\nabla}(u,v)w=\mathfrak S T(T(u,v),w)+\mathfrak S\big((\nabla_uT)(v,w)\big),
\end{equation}
where $\mathfrak S$ is the cyclic sum.

Applying~\eqref{1} to $u,v,w\in \mathcal{D}$ and, making use of property (5) of the torsion of the connection $\nabla$,  we obtain
$$
T(T(u,v),w)= R(R(u,v),w)=0
$$
by property (4). Condition (i) of the definition of the connection $\nabla$ and property (a) of Proposition~\ref{prop:nabla} implies that $R^{\nabla}(u,v)w\in \mathcal{D}$, while the vector
\begin{eqnarray*}
(\nabla_uT)(v,w)&=& - (\nabla_uR )(v,w)
\end{eqnarray*}
belongs to $\mathcal{V}$ by condition (iii). Thus, taking the projection of \eqref{1} to $\calV$,
we obtain
\begin{eqnarray}\label{eq:cyclic}
-\mathfrak S\big((\nabla_uT)(v,w)\big)&=&\mathfrak S\big((\nabla_uR)(v,w)\big)\nonumber
\\
&=&(\nabla_uR)(v,w)+(\nabla_vR)(w,u)+(\nabla_wR)(u,v)=0,
\end{eqnarray}
for any $u,v,w\in \mathcal{D}$.

Since $(\nabla_v R)(v, \cdot) = 0$ for any $v \in \mathcal{D}$, the 3-linear map $(u,v,w)\mapsto (\nabla_vR)(v,w)$ is a vector valued 3-form. It follows by \eqref{eq:cyclic} that $\mathfrak{S} (\nabla_u T)(v,w) = 3 (\nabla_uR)(v,w) =0$. 
\end{proof}

%%%%%%%%%%%%%%%%%%%%%%%%%%%%%%%%%%%%

\section{Projections of constant curvature and geometry}\label{ssec:spaces}

%%%%%%%%%%%%%%%%%%%%%%%%%%%%%%%%%%%%

\subsection{Canonical Riemannian extension}
Let $(M,g^*)$, or equivalently let $(M,\mathcal D,g)$, be a sub-Riemannian manifold and $\pi\colon M\to N$ a submersion. We again suppose that there is a Riemannian metric $g_N$ on $N$ such that $d\pi$ is an isometry from $\calD$ to $TN$ on every fiber. The following question was studied in~\cite{GG}: describe all possible Riemannian metrics $g_M$ on $M$ such that they are extensions of $g$ and the projection on $N$ of any Riemannian geodesic coincides with the projection of sub-Riemannian geodesics. We describe a special extension $g^*_M$ of the cometric $g^*$ which is a cometric for a Riemannian metric $g_M$ such that $g_M$ gives a positive answer to the stated question. 

From the cometric $g^*$, we get an induced degenerate inner product on $\bigwedge^2 T^*M$ by
$$\langle \zeta^1, \zeta^2 \rangle_{g^*} = \sum_{1 \leq i < j \leq n} \zeta^1(v_i, v_j) \zeta^2(v_i, v_j), \qquad \zeta^1, \zeta^2 \in \bigwedge^2 T^*_x M,$$
where $v_1, \dots, v_n$ is an orthonormal basis of $\calD_x$.
\begin{proposition} \label{prop:Rmetric}
Assume that the projections of all sub-Riemannian normal geodesics are curves of constant first geodesic curvature and vanishing second geodesic curvature.

Define $g^*_M\colon T^*M\otimes T^*M\to{\mathbb R}$ by
\begin{equation} \label{eq:g}
\langle\alpha,\beta\rangle_{g^*_M}=\langle\alpha,\beta\rangle_{g^*}+ \frac{2}{n} \langle R^*\alpha,R^*\beta\rangle_{g^*},
\end{equation}
where $R^*\alpha$ is the two-form defined by $R^*\alpha(V,W)=\alpha(R(V,W))$. Then $g_M^*$ is the cometric of a Riemannian metric $g_M$ and projections of its Riemannian geodesics have constant first geodesic curvature and vanishing second geodesic curvature. Finally, all the fibers of $\pi\colon M\to N$ are totally geodesic submanifolds of $M$.
\end{proposition}

\begin{proof}
%Assume that $\langle\alpha,\cdot\rangle_{g^*_M}=0$ for some $\alpha\neq0$. Then $\alpha\in{\rm Ann}({\mathcal D})$ and $R^*\alpha=0$. Then for any vector fields $V,W_1,W_2\in\Gamma({\mathcal D})$ we have
%\begin{align*}
%\alpha([V,R(W_1,W_2)])=&\alpha({\rm pr}_{\mathcal V}[V,R(W_1,W_2)])=\alpha(\nabla_VR(W_1,W_2))\\
%&=\alpha(R(\nabla_VW_1,W_2))+\alpha(R(W_1,\nabla_VW_2))=0.
%\end{align*}

We start the proof by showing that $g^*_M$ is non degenerate, which will imply that it corresponds to a non-degenerate Riemannian metric $g_M$.
Notice that $g_M^*$ being non degenerate is equivalent to the curvature $R$ being surjective on~${\mathcal V}$. To see this, note that $|\alpha|_{g_M^*}=0$ implies $|\alpha|_{g^*}=0$ and $R^*\alpha=0$. If $R^*\alpha=0$ then it follows that ${\rm im}\,R\subset\ker\alpha$. Since ${\mathcal V}\cap\ker\alpha$ has codimension one, there is a vector $v\notin{\rm im}\,R$. We conclude that %${\rm im}\,R\neq{\mathcal V}$, and thus 
$R$ is not surjective. If $R$ is surjective, then $|\alpha|_{g^*_M}\neq 0$ for any $\alpha\neq0$.

It follows from the above that we only need to check that $R$ is a surjective operator. By using the condition that $\mathcal D$ is bracket generating, we aim to prove that 
\begin{equation}\label{eq:br gen}
TM={\mathcal D}+[{\mathcal D},{\mathcal D}]={\mathcal D}\oplus{\rm im}\,R,
\end{equation}
which will imply that ${\rm im}\,R=\mathcal V$ and will show that $R$ is surjective. 

Since  
$$
[\mathcal D,[\mathcal D,\mathcal D]]=[\mathcal D,{\rm pr}_{\mathcal D}[\mathcal D,\mathcal D]]+[\mathcal D,{\rm pr}_{\mathcal V}[\mathcal D,\mathcal D]],
$$
we concentrate only on the term $[\mathcal D,{\rm pr}_{\mathcal V}[\mathcal D,\mathcal D]]$, which is the only term that could be outside $\mathcal{D} + [\mathcal{D}, \mathcal{D}]$. Observe that for any $V,W_1,W_2\in\Gamma(\mathcal D)$ we have
\begin{eqnarray*}
[V,{\rm pr}_{\mathcal V}[W_1,W_2]]&=&[V,R(W_1,W_2)]=
{\rm pr}_{\mathcal V}[V,R(W_1,W_2)]+{\rm pr}_{\mathcal D}[V,R(W_1,W_2)]
\\
&=& \nabla_VR(W_1,W_2) \mod{\mathcal D},
\end{eqnarray*}
by~(iii). Since $\nabla_VR=0$ for any $V\in\Gamma(\mathcal D)$, we see that 
$$[V,R(W_1,W_2)]=R(\nabla_VW_1,W_2)+R(W_1,\nabla_VW_2)\mod{\mathcal D}.$$ Hence we obtain that the distribution ${\mathcal D}+[{\mathcal D},{\mathcal D}]$ is integrable. This and the fact that $\mathcal D$ is bracket generating leads to~\eqref{eq:br gen}. As a byproduct we have also showed that ${\mathcal D}$ has step~2.

Let $g_M$ be the Riemannian metric associated to the non-degenerate cometric $g_M^*$ and let $\nabla^M$ be the corresponding Levi-Civita connection. The Bott connection $\nabla$ of $g_M$ is defined as
\begin{equation}\label{eq:Bott}
\nabla_VW=\left\{
\begin{array}{ll}
{\rm pr}_{\mathcal D}\nabla^M_V W&V,W\in\Gamma({\mathcal D}),\\
{\rm pr}_{\mathcal V}\nabla^M_V W&V,W\in\Gamma({\mathcal V}),\\
{\rm pr}_{\mathcal D}[V,W]&V\in\Gamma({\mathcal V}),W\in\Gamma({\mathcal D}),\\
{\rm pr}_{\mathcal V}[V,W]&V\in\Gamma({\mathcal D}),W\in\Gamma({\mathcal V}).
\end{array}\right.
\end{equation}
Since $\nabla$ satisfies (i)-(iii), we have that $\nabla_VR=0$ for $V\in\Gamma(\calD)$, and then $\nabla g_M=0$. By~\cite[Theorem~2.5]{GG} we know that projections of sub-Riemannian and Riemannian geodesics coincide. Therefore all the projections of the Riemannian geodesics will have constant geodesic curvature. The results of~\cite{GG} also show that 
the fibers of $\pi$ are totally geodesic submanifolds of $M$.
\end{proof}

\begin{remark} \label{re:Normalization}
The constant $\frac{2}{n}$ can be replaced by any other positive constant and the above results would still hold. This particular choice of the constant implies that for any $\alpha \in \Ann(\calD)_x$ with $\pi(x) =y$ and with an orthonormal basis $v_1, \dots, v_n$ of $T_y N$, we obtain
$$| \alpha |_{g_M^*}^2 = \frac{1}{n} \sum_{i,j=1}^n \alpha R(h_x v_i, h_x v_j)^2  = \frac{1}{n} \sum_{i,j=1}^n \langle J_\alpha v_i, v_j\rangle_{g_N}^2 = \frac{1}{n} \sum_{j=1}^n | J_\alpha v_i |^2_{g_N}. $$
Hence, if $|J_\alpha v|_{g_N}$ is constant for any unit length vector $v$, we have $J^2_\alpha = - |\alpha|_{g^*_M}^2 \Id$.
\end{remark}

%%% Section 3.2%%% %%% %%% %%% %%% %%% %%% 

\subsection{Local description}

%%% %%% %%% %%% %%% %%% %%% 

We give the following local description of the sub-Riemannian spaces considered in the previous sections.
\begin{corollary}\label{cor:local}
The conditions of Theorem~\ref{th:main2} are equivalent to the fact that every $x_0$ has a neighbourhood $U$ where there exists an $\mathbb{R}^{m-n}$-valued one-form $\theta = (\theta_1, \dots. \theta_{m-n})$ satisfying
\begin{enumerate}[\rm (a)]
\item \label{local1} $\ker \theta = \calD$;
\item \label{local2} If $v \in \calD_x$ and $w \in \calV_x$, $x \in U$, then $d\theta(v,w) = 0$;
\item \label{local3} For all vector fields $X,Y, Z \in \Gamma(TN)$ we have that on $U$,
$$hX \theta(hY, hZ) = d\theta(h \nabla_X^N Y, hZ) + d\theta(hY, h \nabla^N Z);$$
\item \label{local4} If $V_1, \dots, V_n$ is an orthonormal basis of $\calD|U$ and we define matrices $A_k(x) = (A_{k,ij}(x))$ by $A_{k,ij}(x) = d\theta_k(V_i, V_j)(x)$, then $- A_k(x)^2$ is a positive semi-definite diagonal matrix.
\end{enumerate}
\end{corollary}
\begin{proof}
Since the properties in question are local, it is sufficient to show that they hold for an arbitrary sufficiently small neighborhood $U$ where $\Ann(\calD)$ is trivial. 

Assume first that Theorem~\ref{th:main2} holds. This is equivalent to \eqref{eq:vR} and \eqref{eq:RvRw} holding.
Let $g_M$ be the Riemannian metric constructed in Proposition~\ref{prop:Rmetric}. Note that the property of the fibers of $\pi: M \to N$ being totally geodesic manifolds is equivalent to the relation
\begin{equation} \label{eq:TGequivalence}(\calL_V g_M)(W,W) = 0, \qquad V \in \Gamma(\calD), W \in \Gamma(\calV),\end{equation}
see e.g.~\cite{GG}.
Choose an orthonormal basis $\theta_1, \dots, \theta_{m-n}$ of $\Ann(\calD)$ with respect to $g_M$. Define $\theta = (\theta_1, \dots, \theta_{m-n})$ as a $\mathbb{R}^{m-n}$ valued one-form which clearly satisfy~\eqref{local1} of Corollary~\ref{cor:local}.
We can also write
\begin{equation} \label{eq:Rtheta} R(V,W) = - \sum_{j=1}^{m-n} d\theta_j (V,W) \sharp \theta_j.\end{equation}
where $\sharp\colon T^*M \to TM$ is the identification with respect to $g_M$. Then the conditions of Theorem~\ref{th:main2} can be interpreted in terms of Corollary~\ref{cor:local} as follows.
Condition \eqref{local4} is a reformulation of~\eqref{eq:RvRw}. We will show that \eqref{local2} and \eqref{local3} follow from \eqref{eq:vR}.

Notice that we can restate the property~\eqref{eq:TGequivalence} as property \eqref{local2}. This follows from
\begin{align*} & (\calL_V g_M) (\sharp \theta_i, \sharp \theta_j) = (\calL_V \theta_i)(\sharp \theta_j) - \theta_i (\calL_V \sharp \theta_j) = 2 (\calL_V \theta_i)(\sharp \theta_j) = 2 d\theta_i(V, \sharp \theta_j).
\end{align*}
The property that $R$ is parallel in horizontal directions can then be expressed by relation
\begin{align} \label{eq:btoc}
0 &= \langle \sharp \theta_i, [hX , R(hY, hZ)] - R(h \nabla_X^N Y, hZ) - R(hY, h \nabla_X^N Z)  \rangle_{g_M} \\ \nonumber
& =  - hX d\theta_i( hY, hZ) + d\theta_i(h \nabla_X^N Y, hZ) + d\theta_i(hY, h \nabla_X^N Z)  \rangle_{g_M} \\ \nonumber
& \qquad - \sum_{i=1}^n d\theta_j(hY, hZ) \langle \sharp \theta_i, \calL_{hX} \sharp \theta_j \rangle_{g_M} \\ \nonumber
& =  - hX d\theta_i( hY, hZ) + d\theta_i(h \nabla_X^N Y, hZ) + d\theta_i(hY, h \nabla_X^N Z)  \rangle_{g_M} ,
\end{align}
the latter equality following from property \eqref{local2}.

Conversely, assuming that we have conditions from \eqref{local1} to \eqref{local4} satisfied, we define a Riemannian metric $g_M$ on $M$ such that $g_M | \calD = g$ and such that $\theta_1, \dots , \theta_{m-n}$ is an orthonormal basis of $\Ann(\calD)$. Then \eqref{eq:Rtheta} still holds, and reversing the arguments above, we obtain~\eqref{eq:vR} and~\eqref{eq:RvRw}.
\end{proof}

%%%%%%%%%%%%%%%%%%%%%%%%%%%%%%%

\section{Examples}\label{sec:example}

%%%%%%%%%%%%%%%%%%%%%%%%%%%%%%%

\subsection{Principal bundles} \label{sec:Principal}
For a given Lie group $G$ with Lie algebra $\mathfrak{g}$, let $G \to P \stackrel{\pi}{\to} N$ be a $G$-principal bundle. The action of $G$ is chosen to be on the right, as usual. Write $\calV = \ker d\pi$. For every element $A \in \mathfrak{g}$, we have a canonically associated vector field $\xi_A$ on $P$ defined by
$$\xi_A(p) = \frac{d}{dt} p \cdot e^{tA} |_{t=0}.$$
In fact, we have a vector bundle isomorphism between $N \times \mathfrak{g}$ and $\calV$, given by $(p, A) \mapsto \xi_A(p)$.
\emph{A connection form $\omega$ on $\pi$} is a $\mathfrak{g}$-valued one-form on $P$, that satisfies
$$\omega(v \cdot a) = \Ad(a^{-1}) \omega(v), \qquad \omega(\xi_A(p)) = A,$$
for any $a \in G$, $A \in \mathfrak{g}$, $v \in T_pP$ and $p \in P$.
Through the relation
$\calD = \ker \omega$,
there is a one-to-one correspondence between connection forms $\omega$ and Ehresmann connections $TP = \calD \oplus \calV$ that have the invariance property $\calD_p \cdot a = \calD_{p \cdot a}$.

The form $\Omega = d\omega + \frac{1}{2} [\omega, \omega]$ is called the curvature form of $\omega$. If $R$ is the curvature of~$\calD$, then
\begin{equation}\label{eq:Rxi}
R(V,W) =- \xi_{\Omega(V,W)}.
\end{equation}

The form $\Omega$ can be seen as a two-form on $N$ with values in a vector bundle $\Ad P$. The latter is a vector bundle $\Ad P \to N$ defined as the quotient of $P \times \mathfrak{g}$ under the equivalence relation
$$(p, A) \sim (p \cdot a, \Ad(a^{-1}) A), \qquad p \in P,\quad A \in \mathfrak{g}.$$
Denote the equivalence class of $(p,A)$ by $[p,A]$. A function $s^\wedge: P \to \mathfrak{g}$ is called equivariant if $s^\wedge(p \cdot a) = \Ad(a^{-1}) s^\wedge(p)$, $p \in P$, $a \in G$. There is a one-to-one correspondence between sections $s \in \Gamma(\Ad P)$ and equivariant functions on $P$ given by
$$s(x) = [p ,s^\wedge(p)], \qquad x \in N,\quad p \in P_x.$$
Given a connection form $\omega$, we can define an affine connection $\nabla^\omega$ on $\Ad P$, by defining $\nabla_X^\omega s$ to be the unique section corresponding to the equivariant function $ds^\wedge(hX)$ for $s \in \Gamma(\Ad P)$, $X \in \Gamma(TM)$. In general, a $j$-form $\eta^\wedge$ is called equivariant if
$$\eta^\wedge( v_1 \cdot a, \dots, v_j \cdot a) = \Ad(a^{-1}) \eta^\wedge(v_1, \dots, v_j).$$
Any equivariant $j$-form that vanishes on $\calV$ can be seen as a $j$-form $\eta$ on $N$ with values in $\Ad P$ through the relation
$$\eta(v_1, \dots, v_j) = [p, \eta^\wedge(h_p v_1, \dots, h_p v_j)].$$
The two-form $\Omega$ vanishes on $\calV$ and is equivariant. 

Assume now that $N$ has a Riemannian metric $g^N$ and lift it to a sub-Riemannian metric $g$ on $\calD = \ker \omega$ through $\pi$. If $\nabla$ satisfies the requirements (i)-(iii), then it is simple to verify that
$$[p, \omega((\nabla_{hX} R)(hY, hZ))] =- (\nabla^\omega_X \Omega)(Y, Z).$$
In the formula above, we have simplified the notation by using the symbol $\nabla^\omega$ for the connection induced on $\bigwedge^2 T^*N \otimes \Ad P$ by $\nabla^N$ and $\nabla^\omega$. 

The condition \eqref{eq:J2} can we written as
\begin{equation} \label{PBeq2}\langle \lambda \Omega(v,\cdot) , \lambda\Omega(w, \, \cdot \,) \rangle_{g_N} = 0,\end{equation}
whenever $\lambda \in \mathfrak{g}^*$ and $v$ and $w$ are orthogonal.

In conclusion, projections of geodesics in $(P, \calD, g)$ are always constant curvature curves if and only if $\nabla^\omega \Omega = 0$ and \eqref{PBeq2} are satisfied.

\begin{example}[Orthonormal frame bundle] \label{ex:Frame}
We will consider the sub-Riemannian geometry of the orthonormal frame bundle. For details on these spaces, we refer to \cite[Section~3.3.1]{Gro16} and \cite[Chapter~2.1]{Hsu02}.

 Let $N$ be a Riemannian manifold. Define $\pi\colon\Ort(N) \to N$ as the orthonormal frame bundle of $N$. This is an $\Ort(n)$-principal bundle such that for any $y \in N$
$$\Ort(N)_y = \left\{  f: \mathbb{R}^n \to T_y N \, : \, \text{$f$ is a linear isometry}\right\}.$$
Here, $\mathbb{R}^n$ is equipped with the standard euclidean structure. There is a one-to-one correspondence between maps $f: \mathbb{R}^n \to T_y N$ and choices $f_1, \dots, f_n$ of orthonormal frames of $T_yN$. If $e_1, \dots, e_n$  is the standard basis of $\mathbb{R}^n$, then the correspondence is given by $f_j = f(e_j)$.

Define a subbundle of $T\Ort(N)$ by
\begin{equation} \label{eq:tildeD} \tilde {\calD} = \left\{ \dot f(0) \, : \, \begin{array}{c} f:(-\varepsilon, \varepsilon) \to \Ort(N), \, \pi(f(t)) = \gamma(t), \, \nabla^N_{\dot \gamma} f_j(t) = 0 
\end{array} \right\}.\end{equation}
Then this subbundle is an Ehresmann connection which is invariant under the action of $\Ort(n)$. Hence, it corresponds to a principal curvature form $\omega: T\Ort(N) \to \so(n)$. Let $h_f u$ denote the horizontal lift of $u\in T_yN$, $y\in N$ to $f \in \Ort(N)_y$. If $R^N$ denotes the Riemannian curvature tensor of $N$, the curvature form $\Omega$ of $\omega$ is given as
$$\Omega(h_f u, h_f v) = -\Big(\langle R^N(u, v) f_i, f_j \rangle_{g_N}\Big)_{i,j}, \qquad u,v \in T_y N, \, f \in \Ort(N)_y, \, y \in N,$$
and $\nabla^\omega \Omega =0$ if and only if $\nabla^N R^N = 0$, i.e. if $N$ is a locally symmetric space.

Next, for $r = (r_1, \dots, r_n) \in \mathbb{R}^{n}$, define $f(r) = \sum_{i=1}^n r_i f_i$. Furthermore, for $r, s \in \mathbb{R}^n$, and $v\in T_yN$, if we define
$$J_{r,s}(f) := J_{r^\top \omega(f)( \, \cdot \,) s}.$$
then
$$J_{r,s}(f) v = - \sharp r^\top\Omega(h_fv, \, \cdot \,)s =  \sharp \langle R^N(v, \,\cdot \,)f(r) ,f(s)\rangle_{g_N} =  R^N(f(r), f(s))v.$$
Since $r$ and $s$ were arbitrary and since any $J_\alpha$ can be written as a sum of maps on the above form, it follows that~\eqref{eq:J2} is equivalent to
\begin{equation} \label{eq:R2} R^N(u,v)^2 w = R^N(u,v) R^N(u,v) w =  - |R^N(u,v) w|_{g_N}^2 w, \end{equation}
for unit vectors $u$, $v$ and $w$.

Hence, if we consider the sub-Riemannian manifold $(\Ort(N), \tilde {\calD}, g)$ where the metric $g$ is pulled back of $g_N$ from $N$, then projections of sub-Riemannian normal geodesics have constant first geodesic curvature and vanishing second geodesic curvature if and only if $N$ is a locally symmetric Riemannian manifold and satisfies \eqref{eq:R2}.
\end{example}

\subsection{Complete manifolds}
Assume that $\pi: M \to N$ is a submersion into a Riemannian manifold. Let $\calD$ be an Ehresmann connection on $\pi$ and assume that the conditions of Theorem~\ref{th:main2} holds. Furthermore, assume that the Riemannian metric $g_M$ defined in Proposition~\ref{prop:Rmetric} is complete. Since the fibers of $\pi$ are totally geodesic submanifolds, we can conclude the following from \cite{Her60}.
\begin{enumerate}[\rm (A)]
\item $N$ is a complete manifold.
\item Each fiber $M_y = \pi^{-1}(y)$ with the restricted metric is isometric to the same complete Riemannian manifold $(F, g_F)$.
\item Let $G$ be the isometry group of $F$. Then there exists some principal $G$-bundle $G \to P \to N$ such that $M$ is diffeomorphic to
$$M = (P \times F) /G.$$
Here, the action of $a \in G$ is given by $(p,z) \cdot a = (p \cdot a, a^{-1} z)$, $a \in G$, $p \in P$, $z \in F$.
\item If $\rho: P \times F \to M$ is the quotient map, then there is a principal connection $\omega$ on $P \to N$ such that if we define
$$\tilde {\calD} = \left\{ \frac{d}{dt} ( p(t), z) |_{t=0} \in T(P \times F) \, : \, p: (- \varepsilon, \varepsilon) \to P , \,  \omega(\dot p(t)) = 0 ,\, z \in F\right\},$$
then $\tilde {\calD}$ is an Ehresmann connection on $P \times F \to N$ and $d\rho(\tilde {\calD}) = \calD$.
\item As was observed in \cite[Remark~4.2]{GT16}, the assumption that $\calD$ is bracket-generating implies that $F$ is a homogeneous space, i.e. $F = G/K$ for some closed subgroup $K$ of $G$.
\end{enumerate}
For a vector field $X$ on $N$, let $hX$ and $\tilde hX$ denote its horizontal lifts to $\calD$ and $\tilde {\calD}$, respectively. Let $R$ and $\tilde R$ denote the respective curvatures of $\calD$ and $\tilde {\calD}$ and write $\Omega$ for the curvature form of $\omega$.

Write $\tilde{\calV}$ for the vertical bundle of $P \times F\to N$. Using the fact that $\tilde hX$ and $hX$ are $\rho$-related, we have that
$$R(hX, hY) = \mathrm{pr}_{\calV} \, [hX, hY] = d\rho\,  \mathrm{pr}_{\tilde{\calV}} \, [\tilde h X, \tilde h Y] = d\rho \, \tilde R(hX, hY) = - d\rho \, \xi_{\Omega(\tilde hX, \tilde hY)}.$$
Hence, $R(hX, hY)$ and $- \xi_{\Omega(\tilde hX, \tilde Y)}$ are $\rho$-related vector fields.
We can rewrite~\eqref{eq:vR} as
\begin{align*}
0 & = \mathrm{pr}_{\calV} \, [hX, R(hY, hZ)] - R(h\nabla_X^N Y, hZ) - R(hY, \nabla_X Z) \\
& = - \mathrm{pr}_{\calV}\,d\rho \left( [\tilde hX, \xi_{\Omega(\tilde hY, \tilde hZ)}] - \xi_{\Omega(\tilde h \nabla_X^N Y, \tilde hZ) + \Omega(\tilde h Y, \tilde h\nabla_X^N Z)} \right) \\
& = - d\rho \left( \xi_{\tilde hX \Omega(\tilde hY, \tilde hZ) -\Omega(\tilde h \nabla_X^N Y, \tilde hZ) - \Omega(\tilde h Y, \tilde h\nabla_X^N Z)} \right) = - d\rho \xi_{(\nabla_X^\omega \Omega)(Y,Z)^\wedge},
\end{align*}
where we have used the correspondence between equivariant functions and sections of $\Ad P$. Since the action of the isometry group on $F$ is faithful, the vector field $d\rho \xi_A$ only vanishes if $A = 0$. Hence,  we have that all projections from $M$ to $N$ of normal geodesics are of constant first geodesic curvature and vanishing second geodesic curvature if and only if $\nabla^\omega \Omega = 0$ and condition~\eqref{PBeq2} holds.

\begin{example}[Unit tangent sphere bundle]
Let $N$ be a Riemannian manifold and consider $\pi:SN\to N$ the unit tangent bundle, $$SN = \{ v \in TN \, : \, |v|_{g_N} =1 \}.$$
Define an Ehresmann connection $\calD$ on $\pi$ by
\begin{equation} \label{Dparallel} \calD = \left\{ \dot X(0) \, : \, X: (-\varepsilon, \varepsilon) \to SN, \, \pi(X(t)) =\eta(t), \, \nabla_{\dot \eta}^N X(t) = 0\right\}.
\end{equation}
We can identify $SN$ with the quotient $(\Ort(N) \times S^{n-1})/\Ort(n)$ where we have identified $\Ort(n)$ with the isometry group of $S^{n-1}$. The subbundle $\calD$ is then the image of $\tilde{\calD}$ from Example~\ref{ex:Frame} under the quotient map. Let $g$ be the sub-Riemannian metric on $\calD$ obtained from pulling back the Riemannian metric on $N$. It then follows from Example~\ref{ex:Frame} that the normal geodesics of $(SN, \calD, g)$ have constant curvature curves as projections if and only if $N$ is locally symmetric and \eqref{eq:R2} holds. \end{example}

\begin{remark}
Let $N$ be a Riemannian manifold and let $\pi: TN \to N$ be its tangent bundle. Let $\calD$ be an Ehresmann connection on $\pi$ be defined as in \eqref{Dparallel}. This allows us to define horizontal lifts of vector fields on $N$. However, we also have vertical lifts to vector fields with values in $\calV = \ker d\pi$. For any $X$, we define the vector field $\vl X$ on $TN$ by
$$\vl X(v) = \frac{d}{dt} (v + t X(y)) |_{t=0}, \qquad y \in N, v \in T_yN.$$
\emph{The Sasakian metric} $g_S$ on $TN$ is defined by
$$\langle hX, hY\rangle_{g_S} = \langle X,Y \rangle_{g_N}, \quad \langle \vl X, \vl Y \rangle_{g_S} = \langle X, Y \rangle_{g_N}, \quad \langle hX, \vl Y \rangle_{g_S} = 0.$$

Consider the restriction of $g_S$, $\pi$ and $\calD$ to $SN \subseteq TN$ which we, by slight abuse of notation, will denote by the same symbol. Then it was shown in~\cite{BBNV03,Sak76} that projections of Riemannian geodesics are constant curvature curves if and only if $\nabla^N R^N = 0$. It is interesting that we get the same result here, even though the metric $g_M$ defined in Section~\ref{ssec:spaces} only coincides with the Sasakian metric for the case of $N = S^n$. This result is a consequence of the fact that, equipped with the Sasakian metric, the submersion $\pi: TN \to N$ always has totally geodesic fibers, see e.g. \cite[Theorem~2]{Sak76}.
\end{remark}

\section{Transverse symmetries and $H$-type manifolds} \label{sec:HType}
Let $\pi: M \to N$ be a submersion between a sub-Riemannian and a Riemannian manifold. We assume that hypotheses of Theorem~\ref{th:main2} are fulfilled, therefore~\eqref{eq:vR} and \eqref{eq:RvRw} hold. Following the definition found in~\cite{BK16}, we say that a sub-Riemannian manifold $(M, \calD, g)$ has transverse symmetries if $\calV$ has a basis $V_1, \dots, V_{m-n}$ such that
\begin{equation} \label{TSym} \pr_{\calD} [X, V_i] = 0, \qquad X \in \Gamma(\calD).
\end{equation}
If $\nabla$ is the Bott connection defined in~\eqref{eq:Bott}, then
$$\nabla_v V_j= 0 \qquad \text{for any } v \in \mathcal{D}.$$
For these classes of manifolds, we state the following result. 

For an arbitrary oriented inner product space $\mathbb{V}$, let $\SO(\mathbb{V})$ be the group of orientation-preserving linear isometries of $\mathbb{V}$. If $\eta$ is a loop in $N$ based at $y$, we denote by $\ptr_\eta\colon T_y N \to T_y N$ the $\nabla^N$-parallel transport along $\eta$. We then define the holonomy group at $y$ by
$$\Hol(y) = \{ \ptr_\eta\colon\text{$\eta$ is a loop is based at $y$} \} \subseteq \SO(T_yN).$$
Then $\Hol(y)$ is a Lie subgroup of $\SO(T_yN)$ and if $y_1$ is another point in $N$, then $\Hol(y)$ and $\Hol(y_1)$ are conjugate groups since $N$ is connected. Hence, if $\Hol(y) = \SO(T_yN)$ at one point, then this is true for all points. Moreover any vector in $TN$ can be taken to another vector in $TN$ by parallel transport.
\begin{theorem} \label{th:Htype}
Let $g_M$ be the metric defined as in Proposition~\ref{prop:Rmetric}. Assume that each fiber $M_y = \pi^{-1}(y)$ is connected and that $\Hol(y) = \SO(T_yN)$ for some $y \in M$. Assume finally that the Lie algebra generated by $V_1, \dots, V_{m-n}$ is abelian.

Then
\begin{equation} \label{Htype} J_{\alpha}^2 = - | \alpha |^2_{g^*} \Id ,\end{equation}
for any $\alpha \in T^*M$. Furthermore, $\mathcal{V}$ has an orthonormal basis of infinitesimal symmetries.
\end{theorem}
Manifolds where \eqref{Htype} is satisfied are said to be of $H$-type according to \cite[Section~4]{BK16}. Using polarization, we get that
$$J_\alpha J_\beta + J_\beta J_\alpha = - 2 \langle \alpha, \beta \rangle_{g^*_M} \Id .$$

\begin{proof}
Observe first that since $V_i$ is a parallel vector field along any horizontal curve and any pair of points can be connected by such curves, the function $\langle V_i, V_j \rangle_{g_M}$ is constant on $M$. Hence, we can take a linear combination of these vector fields with constant coefficients to obtain a global orthonormal basis which still satisfies \eqref{TSym}. We will assume from now on that we have chosen $V_1, \dots, V_{m-n}$ orthonormal.

By the definition of the Levi-Civita connection, since $[V_i, V_j] = 0$, we have that $\nabla V_i =0$. Let $\theta_1, \dots, \theta_{m-n}$ be a basis of $\Ann(\calD)$ determined by $\theta_i = \langle V_i, \, \cdot \, \rangle_{g_M} $. Let $\gamma: I \to M$ be a any curve with projection $\eta$ in $N$. Then for any parallel vector field $X(t)$ along $\eta$, we have
$$\frac{d}{dt} | J_{\theta_i(\gamma(t))} X(t) |^2 = 0,$$
from Theorem~\ref{th:main2}.
In particular, if $\gamma(t)$ is a curve in $M_y$, then $|J_{\theta_i(\gamma(t))} v|$ is constant for any $v \in T_y N$. Since $M_y$ is connected, we have that $|J_{\theta_i(x)} v|$ is constant for any $x \in M_y$ and fixed $v \in T_yN$. Furthermore, if $v \in T_{y_0} N$ and $w \in T_{y_1} N$ are two unit vectors, by our assumption on the holonomy, there is a curve $\eta: I \to N$ and a parallel vector field $X$ along $\eta$ with $X(0) = v$ and $X(1) = w$. Hence, if $\gamma$ is a horizontal lift of $\eta$, we have that $| J_{\theta(\gamma(0))} v | = | J_{\theta(\gamma(1))} w |$. In conclusion, there is some constant $c$ such that
$$c = | J_{\alpha} v| \quad \text{for any $(\alpha, v) \in \Ann(\calD_x) \oplus \pi^* T_yN$,\ \ $\pi(x)=y$\ \  with $|\alpha|_{g^*_M} = 1,\ |v|_{g_N} =1$.}$$
The result finally follows by realizing that $c^2 = | \alpha |_{g^*_M}^2 =1$ by Remark~\ref{re:Normalization}.
\end{proof}

\begin{example}
Let $G \to P \to N$ be a $G$-principal bundle over a Riemannian manifold $(N, g_N)$. Consider a curvature form $\omega$ and define $(P, \calD, g)$ as in Section~\ref{sec:Principal}. For any element $A\in \mathfrak{g}$, we have that $\xi_A$ satisfies \eqref{TSym}. Hence, this is a sub-Riemannian manifold with transverse symmetries. Assume that $\mathrm{Hol}(y) = \SO(T_yN)$ for some $y \in N$ and assume that $G$ is abelian. Then the conditions of Theorem~\ref{th:Htype} is satisfied.
\end{example}

\begin{example}
We will end with an example showing that assumption on holonomy is necessary for the result. Consider the Lie algebra spanned by elements $X$, $Y$ and $Z$, with bracket relations
$$[X,Y] = Z, \qquad [X,Z] = [Y,Z] = 0.$$
This algebra is called \emph{the Heisenberg algebra}. Let $M$ be the corresponding simply connected Lie group. We will use the same symbol for elements in the Lie algebra and their corresponding left invariant vector fields. Let $X^*$, $Y^*$ and $Z^*$ denote the corresponding coframe.

Let $\hat M$ be another copy of $M$ with corresponding left invariant basis $\hat X$, $\hat Y$ and $\hat Z$ and use similar notation for its coframe. Let $(\calD, g)$ be a sub-Riemannian structure on $M \times \hat M$ such that $X$, $Y$, $\hat X$ and $\hat Y$ form an orthonormal basis. If we define $K$ as the abelian subgroup with Lie algebra $\mathfrak{k}$ spanned by $Z$ and $\hat Z$, we get a submersion $\pi: M \times \hat M \to N = \mathbb{R}^4 = (M \times \hat M)/K$. If $N$ is equipped with the standard euclidean metric $g_N$, then $d\pi$ is a linear isometry from $\calD$ to $TN$ on every fiber.

The sub-Riemannian manifold $(M \times \hat M, \calD, g)$ satisfies the conditions of Theorem~\ref{th:main2}. Furthermore, $Z$ and $\hat Z$ satisfy \eqref{TSym} and these commute except the conditions for the holonomy. We also see that
$$|J_{Z^*} X| = |J_{Z^*} Y| =1, \quad \text{while} \quad |J_{Z^*} \hat X| = |J_{Z^*} \hat Y| = 0,$$
reflecting the fact that we cannot take vectors tangent to $M$ to vectors tangent to $\hat M$ using parallel transport.
\end{example}


\begin{thebibliography}{99}
\bibitem{BK16} F.~Baudoin, B.~Kim,
{\it The Lichnerowicz-Obata theorem on sub-Riemannian manifolds with transverse symmetries.}
J. Geom. Anal. {\bf 26} (2016), no. 1, 156--170.

\bibitem{BBNV03} J.~Berndt, E.~Boeckx, P.~T.~Nagy, L.~Vanhecke,
{\it Geodesics on the unit tangent bundle.}
Proc. Roy. Soc. Edinburgh Sect. A {\bf 133} (2003), no. 6, 1209--1229. 

\bibitem{B} L. Bianchi, {\it Vorlesungen \"uber Differentialgeometrie}. Leipzig 1899.

\bibitem{Bott} R.~Bott, {\it Lectures on characteristic classes and foliations}, in Lectures on algebraic and differential topology, Springer Lecture Notes in Mathematics, 279, (1972).

\bibitem{CChM}
O.~Calin, D.~C.~Chang, I.~Markina,  
{\it Geometric analysis on H-type groups related to division algebras.} Math. Nachr. {\bf 282} (2009), no. 1, 44--68. 

\bibitem{CMV} D.-C.~Chang, I.~Markina, A.~Vasil'ev, {\it Sub-Riemannian geodesics on the 3-D sphere}. Complex Anal. Oper. Theory {\bf 3} (2009), no. 2, 361--377.

\bibitem{D} J.-G.~Darboux, {\it Th\'eorie g\'en\'erale des surfaces}. Vol. 3. Gauthier-Villars 1887.

\bibitem{GG} M.~Godoy~M., E.~Grong,
{\it Riemannian and Sub-Riemannian Geodesic Flows}.
J. Geom. Anal. {\bf 27} (2017), no. 2, 1260--1273.

\bibitem{Gro16} E.~Grong,
{\it Submersions, Hamiltonian systems, and optimal solutions to the rolling manifolds problem. }
SIAM J. Control Optim. {\bf 54} (2016), no. 2, 536--566. 


\bibitem{GT16} E.~Grong, A.~Thalmaier, {\it Stochastic completeness and gradient representations for sub-Riemannian manifolds}, arXiv:1605.00785

\bibitem{Her60} R.~Hermann,
{\it A sufficient condition that a mapping of Riemannian manifolds be a fibre bundle.}
Proc. Amer. Math. Soc. {\bf 11} (1960), 236--242. 

\bibitem{Hsu02} E.~P.~Hsu,
{\it Stochastic analysis on manifolds.}
Graduate Studies in Mathematics, 38. American Mathematical Society, Providence, RI, 2002.

\bibitem{KMSB} K.A.~Krakowski, L.~Machado, F.~Silva Leite, J.~Batista, {\it A modified Casteljau algorithm to solve interpolation problems on Stiefel manifolds}. J. Comput. Appl. Math. {\bf 311} (2017), 84--99.

\bibitem{Monti}
R.~Monti, 
{\it Regularity results for sub-Riemannian geodesics}. 
Calc. Var. Partial Differential Equations {\bf 49} (2014), no. 1-2, 549--582.

\bibitem{M} 
R.~Montgomery, {\it A Tour of Subriemannian Geometries, Their Geodesics and Applications}, Math. Surveys Monogr., vol. 91, American Mathematical Society, Providence, RI, 2002.

\bibitem{Sak76} S.~Sasaki,
{\it Geodesics on the tangent sphere bundles over space forms.}
J. Reine Angew. Math. {\bf 288} (1976), 106--120. 

\bibitem{S} M.~Spivak, {\it A comprehensive introduction to differential geometry}. Vol. IV. Second edition. Publish or Perish, Inc., Wilmington, Del., 1979.


\end{thebibliography}
\end{document}